\documentclass[a4paper,abstracton,12pt]{scrartcl}

\usepackage[english]{babel}
\usepackage[utf8]{inputenc}
\usepackage[T1]{fontenc}
\usepackage{textcomp}
\usepackage{amsfonts}
\usepackage{amsmath}
\usepackage{amssymb}
\usepackage[amsmath,thmmarks,hyperref,amsthm]{ntheorem}
\usepackage[shortlabels]{enumitem}
\usepackage{multirow}
\usepackage[pdftex]{hyperref}
\usepackage{url}
\usepackage{mathdots}
\usepackage{mathtools}
\usepackage[autostyle]{csquotes}
\usepackage[
	backend=biber,
	style=numeric,
	firstinits = true,
	sorting = nyt,
	maxnames = 8,
	sortlocale = en,
	doi=false
]{biblatex}

\newbibmacro{string+doi}[1]{%
	  \iffieldundef{doi}{#1}{\href{http://dx.doi.org/\thefield{doi}}{#1}}}
	  \DeclareFieldFormat{title}{\usebibmacro{string+doi}{\mkbibemph{#1}}}
	  \DeclareFieldFormat[article]{title}{\usebibmacro{string+doi}{\mkbibquote{#1}}}

\AtNextBibliography{\small}

\usepackage{remreset}
\makeatletter\@removefromreset{footnote}{chapter}\makeatother

\usepackage{ae}

\newcommand{\F}{\mathbb{F}}
\newcommand{\vek}[1]{\mathbf{#1}}

\newcommand{\qbinom}[3]{\genfrac{[}{]}{0pt}{}{#1}{#2}_{#3}}
\newcommand{\qnumb}[2]{[#1]_{#2}}

\DeclareMathOperator{\PG}{PG}
\DeclareMathOperator{\Der}{Der}

\makeatletter
\def\theorem@checkbold{}
\renewtheoremstyle{plain}%
{\item[\hskip\labelsep \theorem@headerfont ##1\ ##2 \theorem@separator]}%
{\item[\hskip\labelsep \theorem@headerfont ##1\ ##2\ {\normalfont(##3)}  \theorem@separator]}
\makeatother

\newtheorem{theorem}{Theorem}
\newtheorem{lemma}{Lemma}[section]
\theoremstyle{definition}
\newtheorem{definition}[lemma]{Definition}
\theoremstyle{remark}
\newtheorem{remark}[lemma]{Remark}

\title{On $\alpha$-points of $q$-analogs of the Fano plane}
\author{
Michael Kiermaier
\thanks{
Universität Bayreuth, Mathematisches Institut, D-95440 Bayreuth, Germany
\newline
email:~\texttt{michael.kiermaier@uni-bayreuth.de}
\newline
homepage:~\url{http://www.mathe2.uni-bayreuth.de/michaelk/}
}\\
\emph{Universität Bayreuth}
}

\begin{document}
\maketitle
\begin{abstract}
Arguably, the most important open problem in the theory of $q$-analogs of designs is the question for the existence of a $q$-analog $D$ of the Fano plane.
It is undecided for every single prime power value $q \geq 2$.

A point $P$ is called an $\alpha$-point of $D$ if the derived design of $D$ in $P$ is a geometric spread.
In 1996, Simon Thomas has shown that there must always exist at least one non-$\alpha$-point.
For the binary case $q = 2$, Olof Heden and Papa Sissokho have improved this result in 2016 by showing that the non-$\alpha$-points must form a blocking set with respect to the hyperplanes.

In this article, we show that a hyperplane consisting only of $\alpha$-points implies the existence of a partiton of the symplectic generalized quadrangle $W(q)$ into spreads.
As a consequence, the statement of Heden and Sissokho is generalized to all primes $q$ and all even values of $q$.
\end{abstract}

\section{Introduction}
Due to the connection to network coding, the theory of subspace designs has gained a lot of interest recently.
Subspace designs are the $q$-analogs of combinatorial designs and arise by replacing the subset lattice of the finite ambient set $V$ by the subspace lattice of a finite ambient vector space $V$.
Arguably the most important open problem in this field is the question for the existence of a $q$-analog of the Fano plane, which is a subspace design with the parameters $2$-$(7,3,1)_q$.
This problem has already been stated in 1972 by Ray-Chaudhuri~\cite[Problem~28]{Berge-RayChaudhuri-1974-LNMa411:278-287}.
Despite considerable investigations, its existence remains undecided for every single order $q$ of the base field.

It has been shown that the smallest instance $q=2$, the binary $q$-analog of a Fano plane, can have at most a single nontrivial automorphism \cite{Braun-Kiermaier-Nakic-2016-EuJC51:443-457,Kiermaier-Kurz-Wassermann-2018-DCC86[2]:239-250}.
A $q$-analog of the Fano plane would be the largest possible $[7,4;3]_q$ constant dimension subspace code.
However, the hitherto best known sizes of such constant dimension subspace codes still leave considerable gaps, namely $333$ vs. $381$ in the binary case \cite{Heinlein-Kiermaier-Kurz-Wassermann-2019-AiMoC13[3]:457-475} and $6978$ vs. $7651$ in the ternary case \cite{Honold-Kiermaier-2016}.%
\footnote{The $[7,4;3]_3$ code of size $6977$ in \cite{Honold-Kiermaier-2016} can be extended trivially by a further codeword.}

Another approach has been the investigation of the derived designs of a putative $q$-analog $D$ of the Fano plane.
A derived design exists for each point $P\in\PG(6,q)$ and is always a $q$-design with the parameters $1$-$(6,2,1)_q$, which is the same as a line spread of $\PG(5,q)$.
Following the notation of \cite{Heden-Sissokho-2016-ArsComb124:161-164}, a point $P$ is called an \emph{$\alpha$}-point of $D$ if the derived design in $P$ is the geometric spread, which is the most symmetric and natural one among the $131044$ isomorphism types \cite{Mateva-Topalova-2009-JCD17[1]:90-102} of such spreads.
For highest possible regularity, one would expect all points to be $\alpha$-points.
However, this has been shown to be impossible, so there must always be at least one non-$\alpha$-point of $D$ \cite{Thomas-1996-GeomDed63[3]:247-253}.
For the binary case $q=2$, this result has been improved to the statement that each hyperplane contains at least a non-$\alpha$-point {Heden-Sissokho-2016-ArsComb124:161-164},
In other words, the non-$\alpha$-points of of a binary $q$-analog of the Fano plane form a blocking set with respect to the hyperplanes.

In this article, $\alpha$-points will be investigated for general values of $q$, which will lead to the following theorem.
\begin{theorem}
	\label{thm:partition}
	Let $D$ be a $q$-analog of the Fano plane and assume that there exists a hyperplane $H$ such that all points of $H$ are $\alpha$-points of $D$.
	Then the following equivalent statements hold:
	\begin{enumerate}[(a)]
		\item\label{thm:partition:w} The line set of the symplectic generalized quadrangle $W(q)$ is partitionable into spreads.
		\item\label{thm:partition:q} The point set of the parabolic quadric $Q(4,q)$ is partitionable into ovoids.
	\end{enumerate}
\end{theorem}

As a consequence, we will get the following generalization of the result of \cite{Heden-Sissokho-2016-ArsComb124:161-164}.

\begin{theorem}
	\label{thm:alpha_points_block}
	Let $D$ be a $q$-analog of the Fano plane and $q$ be prime or even.
	Then each hyperplane contains a non-$\alpha$-point.
	In other words, the non-$\alpha$-points form a blocking set with respect to the hyperplanes.
\end{theorem}

\section{Preliminaries}
Throughout the article, $q \neq 1$ is a prime power and $V$ is a vector space over $\F_q$ of finite dimension $v$.

\subsection{The subspace lattice}
For simplicity, a subspace $U$ of $V$ of dimension $\dim_{\F_q}(U) = k$ will be called a \emph{$k$-subspace}.
The set of all $k$-subspaces of $V$ is called the \emph{Graßmannian} and will be denoted by $\qbinom{V}{k}{q}$.
Picking the \enquote{best of two worlds}, we will prefer the algebraic dimension $\dim_{\F_q}(U)$ over the geometric dimension $\dim_{\F_q}(U) - 1$, but we will otherwise make heavy use of geometric notions, such as calling the $1$-subspaces of $V$ \emph{points}, the $2$-subspaces \emph{lines}, the $3$-subspaces \emph{planes}, the $4$-subspaces \emph{solids} and the $(v-1)$-subspaces \emph{hyperplanes}.
In fact, the \emph{subspace lattice} $\mathcal{L}(V)$ consisting of all subspaces of $V$ ordered by inclusion is nothing else than the finite projective geometry $\PG(v-1,q) = \PG(V)$.%
\footnote{In established symbols like $\PG(v-1,q)$, the geometric dimension $v-1$ is not altered.}
There are good reasons to consider the subset lattice as a subspace lattice over the unary \enquote{field} $\F_1$ \cite{Cohn-2004-AMM111[6]:487-495}.

The number of all $k$-subspaces of $V$ is given by the Gaussian binomial coefficient
\[
\#\qbinom{V}{k}{q}
= \qbinom{v}{k}{q}
= \begin{cases}
	\frac{(q^v-1)\cdots(q^{v-k+1}-1)}{(q^k-1)\cdots(q-1)} & \text{if } k\in\{0,\ldots,v\}\text{;}\\
	0 & \text{otherwise.}
\end{cases}
\]
The Gaussian binomial coefficient $\qbinom{v}{1}{q}$ is also known as the $q$-analog of the number $v$ and will be abbreviated as $\qnumb{v}{q}$.

For $S \subseteq \mathcal{L}(V)$ and $U,W\in\mathcal{L}(V)$, we will use the abbreviations
\begin{align*}
S|_U & = \{B\in S \mid U \leq B\}\text{,} \\
S|^W & = \{B\in S\mid B\leq W\}\quad\text{and} \\
S|_U^W & = \{B\in S\mid U \leq B\leq W\}\text{.}
\end{align*}
For a point $P$ in a plane $E$, the set of all lines in $E$ passing through $P$ is known as a \emph{line pencil}.

The subspace lattice $\mathcal{L}(V)$ is isomorphic to its dual, which arises from $\mathcal{L}(V)$ by reversing the order.
Fixing a non-degenerate bilinear form $\beta$ on $V$, a concrete isomorphism is given by $U \mapsto U^\perp$, where $U^\perp = \{\vek{x}\in V \mid \beta(\vek{x},\vek{u}) = 0\text{ for all }\vek{u}\in U\}$.
When addressing the dual of some geometric object in $\PG(V)$, we mean its (element-wise) image under this map.
Up to isomorphism, the image does not depend on the choice of $\beta$.

\subsection{Subspace designs}
\label{subsect:designs}
\begin{definition}
Let $t,v,k$ be integers with $0 \leq t \leq k\leq v-t$ and $\lambda$ another positive integer.
A set $D \subseteq\qbinom{V}{k}{q}$ is called a $t$-$(v,k,\lambda)_q$ \emph{subspace design} if each $t$-subspace of $V$ is contained in exactly $\lambda$ elements (called \emph{blocks}) of $D$.
In the important case $\lambda = 1$, $D$ is called a \emph{$q$-Steiner system}.
\end{definition}

The earliest reference for subspace designs is \cite{Cameron-1974}.
However, the idea is older, since it is stated that ``Several people have observed that the concept of a $t$-design can be generalised [...]''.
They have also been mentioned in a more general context in \cite{Delsarte-1976-JCTSA20[2]:230-243}.
The first nontrivial subspace designs with $t \geq 2$ has been constructed in \cite{Thomas-1987-GeomDed24[2]:237-242}, and the first nontrivial Steiner system with $t \geq 2$ in \cite{Braun-Etzion-Ostergard-Vardy-Wassermann-2016-ForumMathPi:e7}.
An introduction to the theory of subspace designs can be found at \cite{Braun-Kiermaier-Wassermann-2018-COST1}, see also \cite[Day~4]{Suzuki-1989-Designs}.

Subspace designs are interlinked to the theory of network coding in various ways.
To this effect we mention the recently found $q$-analog of the theorem of Assmus and Mattson \cite{Byrne-Ravagnani-SIAMJDM33[3]:1242-1260}, and that a $t$-$(v,k,1)_q$ Steiner systems provides a $(v,2(k-t+1);k)_q$ constant dimension network code of maximum possible size.

Classical combinatorial designs can be seen as the limit case $q=1$ of subspace designs.
Indeed, quite a few statements about combinatorial designs have a generalization to subspace designs, such that the case $q = 1$ reproduces the original statement \cite{Kiermaier-Laue-2015-AiMoC9[1]:105-115, Kiermaier-Pavcevic-2015-JCD23[11]:463-480, Nakic-Pavcevic-2015-DCC77[1]:49-60, Braun-Kiermaier-Kohnert-Laue-2017-JCTSA147:155-185}.

One example of such a statement is the following \cite[Lemma~4.1(1)]{Suzuki-1990-EuJC11[6]:601-607}, see also \cite[Lemma~3.6]{Kiermaier-Laue-2015-AiMoC9[1]:105-115}:
If $D$ is a $t$-$(v, k, \lambda)_q$ subspace design, then $D$ is also an $s$-$(v,k,\lambda_s)_q$ subspace design for all $s\in\{0,\ldots,t\}$, where
\[
	\lambda_s \coloneqq \lambda \frac{\qbinom{v-s}{t-s}{q}}{\qbinom{k-s}{t-s}{q}}.
\]
In particular, the number of blocks in $D$ equals
\[
\#D = \lambda_0 = \lambda \frac{\qbinom{v}{t}{q}}{\qbinom{k}{t}{q}}.
\]
So, for a design with parameters $t$-$(v, k, \lambda)_q$, the numbers $\lambda_s$ necessarily are integers for all $s\in\{0,\ldots,t\}$ (\emph{integrality conditions}).
In this case, the parameter set $t$-$(v,k,\lambda)_q$ is called \emph{admissible}.
It is further called \emph{realizable} if a $t$-$(v,k,\lambda)_q$ design actually exists.
The smallest admissible parameters of a $q$-analog of a Steiner system are $2$-$(7,3,1)_q$, which are the parameters of the $q$-analog of the Fano plane.
This explains the significance of the question of its realizability.

The numbers $\lambda_i$ can be refined as follows.
Let $i,j$ be non-negative integers with $i + j \leq t$ and let $I\in\qbinom{V}{i}{q}$ and $J\in\qbinom{V}{v-j}{q}$.
By \cite[Lemma~4.1]{Suzuki-1990-EuJC11[6]:601-607}, see also \cite[Lemma 5]{Braun-Kiermaier-Wassermann-2018-COST1}, the number
\[
	\lambda_{i,j} \coloneqq \# D|_I^J = \lambda\frac{\qbinom{v-i-j}{k-i}{q}}{\qbinom{v-t}{k-t}{q}}
\]
only depends on $i$ and $j$, but not on the choice of $I$ and $J$.
The numbers $\lambda_{i,j}$ are important parameters of a subspace design.
A further generalization is given by the intersection numbers in \cite{Kiermaier-Pavcevic-2015-JCD23[11]:463-480}.

A nice way to arrange the numbers $\lambda_{i,j}$ is the following triangle form.
\[
	\begin{array}{cccccccccc}
	& & & & \lambda_{0,0} \\
	& & & \lambda_{1,0} & & \lambda_{0,1} \\
	& & \lambda_{2,0} & & \lambda_{1,1} & & \lambda_{0,2} \\
	&\iddots & & \iddots & & \ddots & & \ddots \\
	\lambda_{t,0} & & \lambda_{t-1,1} & & \hdots & & \lambda_{1,t-1} & & \lambda_{0,t}
	\end{array}
\]
For a $q$-analog of the Fano plane, we get:
\begin{align*}
	\begin{array}{ccccc}
	& \multicolumn{3}{c}{\lambda_{0,0} = q^8 + q^6 + q^5 + q^4 + q^3 + x^2 + 1} \\
	& \lambda_{1,0} = q^4 + q^2 + 1 & & \lambda_{0,1} = q^5 + q^3 + q^2 + 1 \\
	\lambda_{2,0} = 1 & & \lambda_{1,1} = q^2 + 1 & & \lambda_{0,2} = q^2 + 1 \\
	\end{array}
\end{align*}

As a consequence of the numbers $\lambda_{i,j}$, the \emph{dual} design $D^\perp = \{B^\perp \mid B\in D\}$ is a subspace design with the parameters
\[
t\text{-}\Big(v,v-k,\frac{\qbinom{v-t}{k}{q}}{\qbinom{v-t}{k-t}{q}}\Big)\text{.}
\]

For a point $P \leq V$, the \emph{derived} design of $D$ in $P$ is
\[
	\Der_P(D) = \{ B/P \mid B \in D|_P\}\text{.}
\]
By \cite{Kiermaier-Laue-2015-AiMoC9[1]:105-115}, $\Der_P(D)$ is a subspace design with the parameters $(t-1)$-$(v-1,k-1,\lambda)_q$.
In the case of a $q$-analog of the Fano plane, $\Der_P(D)$ has the parameters $1$-$(6,2,1)$.

\subsection{Spreads}

A $1$-$(v,k,1)_q$ Steiner system $\mathcal{S}$ is just a partition of the point set of $V$ into $k$-subspaces.
These objects are better known under the name $(k-1)$-\emph{spread} and have been investigated in geometry well before the emergence of subspace designs.
A $1$-spread is also called \emph{line spread}.

A set $\mathcal{S}$ of $k$-subspaces is called a \emph{partial $(k-1)$-spread} if each point is covered by at most one element of $\mathcal{S}$.
The points not covered by any element are called \emph{holes}.
A recent survey on partial spreads is found in \cite{Honold-Kiermaier-Kurz-2018-COST}.

The parameters $1$-$(v,k,1)_q$ are admissible if and only $v$ is divisible by $k$.
In this case, spreads do always exist \cite[\S VI]{Segre-1964-AnnDMPA64[1]:1-76}.
An example can be constructed via field reduction:
We consider $V$ as a vector space over $\F_{q^k}$ and set $\mathcal{S} =  \qbinom{V}{1}{q^k}$.
Switching back to vector spaces over $\F_q$, the set $\mathcal{S}$ is a $(k-1)$-spread of $V$, known as the \emph{Desarguesian} spread.

A $(k-1)$-spread $\mathcal{S}$ is called \emph{geometric} or \emph{normal} if for two distinct blocks $B,B'\in\mathcal{S}$, the set $\mathcal{S}|^{B + B'}$ is always a $(k-1)$-spread of $B + B'$.
In other words, $\mathcal{S}$ is geometric if every $2k$-subspace of $V$ contains either $0$, $1$ or $\frac{\qnumb{2k}{q}}{\qnumb{k}{q}} = q^k + 1$ blocks of $\mathcal{S}$.
It is not hard to see that the Desarguesian spread is geometric.
In fact, it follows from \cite[Theorem~2]{Barlotti-Cofman-1974-AmSUHamb40:231-241} that a $(k-1)$-spread is geometric if and only if it is isomorphic to a Desarguesian spreads.

The derived designs of a $q$-analog of the Fano plane $D$ are line spreads in $\PG(5,2)$.
These spreads have been classified in \cite{Mateva-Topalova-2009-JCD17[1]:90-102} into $131044$ isomorphism types.
The most symmetric one among these spreads is the Desarguesian spread.
Following the notation of \cite{Heden-Sissokho-2016-ArsComb124:161-164}, a point $P$ is called an \emph{$\alpha$}-point of the $q$-analog of the Fano plane $D$ if the derived design in $P$ is the geometric spread.

\subsection{Generalized quadrangles}
\label{subsect:gq}

\begin{definition}
A \emph{generalized quadrangle} is an incidence structure $Q = (\mathcal{P},\mathcal{L},I)$ with a non-empty set of \emph{points} $\mathcal{P}$, a non-empty set of \emph{lines} $\mathcal{L}$, and an incidence relation $I \subseteq \mathcal{P}\times\mathcal{L}$ such hat
\begin{enumerate}[(i)]
	\item Two distinct points are incident with at most a line.
	\item Two distinct lines are incident with at most one point.
	\item For each non-incident point-line-pair $(P,L)$ there is a unique incident point-line-pair $(P',L')$ with $P\mathrel{I} L'$ and $P' \mathrel{I} L$.
\end{enumerate}
\end{definition}
Generalized quadrangles have been introduced in the more general setting of generalized polygons in \cite{Tits-1959-InstHautesEtudesSciPublMath2:13-60}, as a tool in the theory of finite groups.

A generalized quadrangle $Q = (\mathcal{P},\mathcal{L},I)$ is called \emph{degenerate} if there is a point $P$ such that each point of $Q$ is incident with a line through $P$.
If each line of $Q$ is incident with $t+1$ points, and each point is incident with $s+1$ lines, we say that $Q$ is of \emph{order} $(s,t)$.
The \emph{dual} $Q^\perp$ arises from $Q$ by interchanging the role of the points and the lines.
It is again a generalized quadrangle.
Clearly, $(Q^\perp)^\perp = Q$, and $Q$ is of order $(s,t)$ if and only if $Q^\perp$ is of order $(t,s)$.

Furthermore, $Q$ is said to be \emph{projective} if it is embeddable in some Desarguesian projective geometry.
This means that there is a Desarguesian projective geometry with point set $\mathcal{P}'$, line set $\mathcal{L}'$, and point-line incidence relation $I'$ such that $\mathcal{P}\subseteq\mathcal{P}'$, $\mathcal{L}\subseteq\mathcal{L'}$ and for all $(P,L)\in\mathcal{P}\times \mathcal{L}$ we have $P\mathrel{I} L$ if and only if $P' \mathrel{I'} L'$.
The non-degenerate finite projective generalized quadrangles have been classified in \cite[Th.~1]{Buekenhout-Lefevre-1974-ArchMath25[1]:540-552}, see also \cite[{{4.4.8.}}]{Payne-Thas-2009-FiniteGeneralizedQuadrangles2nd}.
These are exactly the so-called \emph{classical generalized quadrangles} which are associated to a quadratic form or a symplectic or Hermitean polarity on the ambient geometry, see \cite[{{3.1.1.}}]{Payne-Thas-2009-FiniteGeneralizedQuadrangles2nd}.

In this article, two of these classical generalized quadrangles will appear.
\begin{enumerate}[(i)]
\item The symplectic generalized quadrangle $W(q)$ \cite[3.1.1~(iii)]{Payne-Thas-2009-FiniteGeneralizedQuadrangles2nd}, consisting of the set of points of $\PG(3,q)$ together with the totally isotropic lines with respect to a symplectic polarity.
Taking the geometry as $\PG(\F_q^4)$, the symplectic polarity can be represented by the alternating bilinear form $\beta(\vek{x},\vek{y}) = x_1 y_2 - x_2 y_1 + x_3 y_4 - x_4 y_3$.
The configuration of the lines $\mathcal{L}$ in $\PG(3,q)$ is also known as a (general) linear complex of lines, see  \cite[3.1.1~(iii)]{Payne-Thas-2009-FiniteGeneralizedQuadrangles2nd} or~\cite[Th.~15.2.13]{Hirschfeld-1985-FiniteProjectiveSpacesOfThreeDimensions}.
Under the Klein correspondence, $\mathcal{L}$ is a non-tangent hyperplane section of the Klein quadric.
\item
The second one is the parabolic quadric $Q(4,q)$, whose points $\mathcal{P}$ are the zeroes of a parabolic quadratic form in $\PG(4,q)$, and whose lines are all the lines contained in $\mathcal{P}$.
Taking the geometry as $\PG(\F_q^5)$, the parabolic quadratic form can be represented by $q(\vek{x},\vek{y}) = x_1 x_2 + x_3 x_4 + x_5^2$.
\end{enumerate}
Both $W(q)$ and $Q(4,q)$ are of order $(q,q)$.
By \cite[{{3.2.1}}]{Payne-Thas-2009-FiniteGeneralizedQuadrangles2nd} they are duals of each other, meaning that $W(q)^\top \cong Q(4,q)$. 

Let $Q = (\mathcal{P},\mathcal{L},I)$ be a generalized quadrangle.
As in projective geometries, a set $\mathcal{S} \subseteq \mathcal{L}$ is called a \emph{spread} of $Q$ if each point of $Q$ is incident with a unique line in $\mathcal{S}$.
Dually, a set $\mathcal{O} \subseteq \mathcal{P}$ is called an \emph{ovoid} of $Q$ if each line of $Q$ is incident with a unique point in $\mathcal{O}$.
Clearly, the spreads of $Q$ bijectively correspond to the ovoids of $Q^\perp$.
This already shows the equivalence of Part~\ref{thm:partition:w} and~\ref{thm:partition:q} in Theorem~\ref{thm:partition}.

\section{Proof of the theorems}
For the remainder of the article, we fix $v = 7$ and assume that $D \subseteq \qbinom{V}{7}{q}$ is a $q$-analog of the Fano plane.
The numbers $\lambda_{i,j}$ are defined as in Section~\ref{subsect:designs}.

By the design property, the intersection dimension of two distinct blocks $B,B'\in D$ is either $0$ or $1$.
So by the dimension formula, $\dim(B + B') \in \{5,6\}$.
Therefore two distinct blocks contained in a common $5$-space always intersect in a point.
Moreover, a solid $S$ of $V$ contains either a single block of no block at all.
We will call $S$ a \emph{rich solid} in the former case and a \emph{poor solid} in the latter.

\begin{remark}
By \cite[Remark~4.2]{Kiermaier-Pavcevic-2015-JCD23[11]:463-480}, the poor solids form a dual $2$-$(7,3,q^4)_q$ subspace design.
By the above discussion, the $\lambda_{0,2} = q^2 + 1$ blocks in any $5$-subspace $F$ form dual partial spread in $F$.
The poor solids contained in $F$ are exactly the holes of that partial spread.
\end{remark}

We will call a $5$-subspace $F$ such that all the $\lambda_{0,2} = q^2 + 1$ blocks in $F$ pass through a common point $P$ a \emph{$\beta$-flat} with focal point $P$.

\begin{lemma}
	The focal point of a $\beta$-flat is uniquely determined.
\end{lemma}

\begin{proof}
	Assume that $P \neq Q$ are focal points of a $\beta$-flat $F$.
	Then all $\lambda_{0,2} = q^2 + 1 > 1$ blocks in $F$ pass through the line $P + Q$, contradicting the Steiner system property.
\end{proof}

\begin{lemma}
	\label{lem:beta_flat_unique}
	Let $H$ be a hyperplane and $P$ a point in $H$.
	Then $P$ is the focal point of at most one $\beta$-flat in $H$.
\end{lemma}

\begin{proof}
	There are $\lambda_{1,1} = q^2 + 1$ blocks in $H$ passing through $P$.
	For any $\beta$-flat $F < H$ with focal point $P$, all these blocks are contained in $F$.

	Now assume that there are two such $\beta$-flats $F \neq F'$.
	Then the $q^2 + 1 > 1$ blocks in $D|_P^H$ are contained in $F \cap F'$.
	This is a contradiction, since $\dim(F \cap F') \leq 4$ and any solid contains at most a single block.
\end{proof}

\begin{lemma}
	\label{lem:beta_flat}
	Let $F\in\qbinom{V}{5}{q}$ be a $\beta$-flat with focal point $P$.
	\begin{enumerate}[(a)]
		\item\label{lem:beta_flat:modspread} Each point $Q \neq P$ is covered by a unique block in $F$. \\
		(In other words: $D|^F / P$ is a line spread of $F/P \cong \PG(3,q)$)
		\item\label{lem:beta_flat:poor} A solid $S$ of $F$ is poor if and only if it does not contain $P$.
		\item\label{lem:beta_flat:poorspread} For all poor solids $S$ of $F$, the set $\{B \cap S \mid B\in D|^F\}$ is a line spread of $S$.
	\end{enumerate}
\end{lemma}

\begin{proof}
Part~\ref{lem:beta_flat:modspread}:
As the blocks in $D|^F$ intersect each other only in the point $P$, the number of points in $\qbinom{F}{1}{q}\setminus \{P\}$ covered by these blocks is $(q^2 + 1)(\qbinom{3}{1}{q} - 1) = q^4 + q^3 + q^2 + q = \qbinom{5}{1}{q}-1$.
Therefore, all points in $F$ different from $P$ are covered by a single block in $D_F$.

Part~\ref{lem:beta_flat:poor}:
The number of solids in $F$ containing one of the $q^2 + 1$ blocks in $F$ is $(q^2 + 1)\cdot \qbinom{5-3}{4-3}{q} = (q^2 + 1)(q + 1) = q^3 + q^2 + q + 1$.%
\footnote{Remember that a solid cannot contain $2$ blocks.}
These solids are rich.
Moreover, the $q^4$ solids in $F$ not containing $P$ do not contain a block, so they are poor.
As $q^4 + (q^3 + q^2 + q + 1) = \qbinom{5}{4}{q}$ is already the total number of solids in $F$, the poor solids in $F$ are precisely those not containing $P$.

Part~\ref{lem:beta_flat:poorspread}:
Let $S$ be a poor solid of $F$.
For every block $B$ in $F$ we have $\dim(B \cap S) \leq 2$ as $S$ is poor, and moreover $\dim(B \cap S) \geq \dim(B) + \dim(S) - \dim(F) = 3 + 4 - 5 = 2$ by the dimension formula.
So for all blocks $B$ in $F$ we get that $B + S = F$ and $B \cap S$ is a line.
By parts~\ref{lem:beta_flat:modspread} and~\ref{lem:beta_flat:poor}, every point of the poor solid $S$ is contained in a unique block in $F$.
Hence $\{B \cap S \mid B\in D \text{ and }B + S = F\}$ is a line spread of $S$.
\end{proof}

\begin{lemma}
	\label{lem:alpha_to_beta}
	Let $B,B'$ be two distinct blocks which intersect in an $\alpha$-point $P$.
	Then $P$ is the focal point of the $\beta$-flat $B + B'$.
\end{lemma}

\begin{proof}
	Since $P = B \cap B'$ is a point, $F = B + B'$ is a $5$-subspace.
	Since $P$ is an $\alpha$-point, we have that $\{B'' / P \mid B''\in D|^F_P\}$ is a line spread of $F / P \cong \F_q^4$.
	Such a line spread contains $\frac{[4]_q}{[2]_q} = q^2 + 1$ lines, so $F$ contains $q^2 + 1$ blocks passing through $P$.
	However, the total number of blocks contained in $F$ is only $\lambda_{0,2} = q^2 + 1$, so $F$ is a $\beta$-flat with focal point $P$.
\end{proof}

\begin{lemma}
	\label{lem:4flat_to_point}
	Let $F$ be a $5$-subspace such that all points of $F$ are $\alpha$-points.
	Then $F$ is a $\beta$-flat.
\end{lemma}

\begin{proof}
	The $5$-subspace $F$ contains $\lambda_{0,2} = q^2 + 1 > 1$ blocks.
	Let $B, B'$ be two distinct blocks in $F$.
	Then $P = B\cap B'$ is a point and $F = B + B'$.
	By the precondition, $P$ is an $\alpha$-point, so by Lemma~\ref{lem:alpha_to_beta}, $P$ is the focal point of the $\beta$-flat $F$.
\end{proof}

\begin{remark}
	The statement of Lemma~\ref{lem:4flat_to_point} is still true if $F$ contains a single non-$\alpha$-point $Q$.
	Then either all blocks contained in $F$ pass through $Q$, or there are two distinct blocks $B$, $B'$ in $F$ such that $P = B\cap B' \neq Q$.
	In the latter case, all blocks pass through the $\alpha$-point $P$ as in the proof of Lemma~\ref{lem:4flat_to_point}.
\end{remark}

\begin{lemma}
	\label{lem:6-1-chain}
	Let $H$ be a hyperplane and $P$ an $\alpha$-point contained in $H$.
	Then $H$ contains a unique $\beta$-flat whose focal point is $P$.
\end{lemma}

\begin{proof}
	There are $\lambda_{1,1} = q^2 + 1 > 1$ blocks in $H$ containing $P$.
	Let $B,B'\in D|^H_P$.
	Then $P = B \cap B'$.
	By Lemma~\ref{lem:alpha_to_beta}, the $\alpha$-point $P$ is the focal point of the $\beta$-flat $F = B + B'$.
	By Lemma~\ref{lem:beta_flat_unique}, the $\beta$-flat $F$ is unique.
\end{proof}

Now we fix a hyperplane $H$ of $V$ and assume that all its points are $\alpha$-points.

By Lemma~\ref{lem:4flat_to_point}, every $5$-subspace $F$ is a $\beta$-flat.
We denote its unique focal point by $\alpha(F)$.
Moreover by Lemma~\ref{lem:6-1-chain}, each point $P$ of $H$ is the focal point of a unique $\beta$-flat $F$ in $H$.
We will denote this $\beta$-flat by $\beta(P)$.
Clearly, the mappings
\[
\alpha : \qbinom{H}{5}{q} \to \qbinom{H}{1}{q}
\quad\text{and}\quad
\beta : \qbinom{H}{1}{q} \to \qbinom{H}{5}{q}
\]
are inverse to each other.
So they provide a bijective correspondence between the points and the $5$-subspaces of $V$.

\begin{lemma}
	\label{lem:PBF}
	Let $B$ be a block in $H$.
	\begin{enumerate}[(a)]
		\item\label{lem:PBF:PB} For all points $P$ of $B$, $B \leq \beta(P)$.
		\item\label{lem:PBF:BF} For all $5$-subspaces $F$ in $H$ containing $B$, $\alpha(F) \leq B$.
	\end{enumerate}
\end{lemma}

\begin{proof}
	For part~\ref{lem:PBF:PB}, let $P$ be a point on $B$.
	There are $\lambda_{1,1} = q^2 + 1$ blocks in $H$ passing through $P$, which equals the number $\lambda_{0,2}$ of blocks in $\beta(P)$ (which all pass through $P$).
	Therefore, $B \leq \beta(P)$.
	
	For part~\ref{lem:PBF:BF}, let $F$ be a $5$-subspace containing $B$.
	All blocks in $F$ pass through its focal point $\alpha(F)$.
\end{proof}

By Lemma~\ref{lem:beta_flat:poor}, each $5$-subspace and a fortiori the hyperplane $H$ contains a poor solid $S$.
This poor solid $S$ is fixed for the remainder of this article.

The set of $\qbinom{6-4}{5-4}{q} = q+1$ intermediate $5$-subspaces $F$ with $S < F < H$ will be denoted by $\mathcal{F}$.
For each $F\in\mathcal{F}$, the set $\mathcal{L}_F := \{B \cap S \mid B\in D\text{ and }B\in D|^F\}$ is a line spread of $S$ by Lemma~\ref{lem:beta_flat}\ref{lem:beta_flat:poorspread}.

\begin{lemma}
\label{lem:LF_disjoint}
The line spreads $\mathcal{L}_F$ with $F\in\mathcal{F}$ are pairwise disjoint.
\end{lemma}

\begin{proof}
Let $L\in \mathcal{L}_F\cap \mathcal{L}_{F'}$ with $F,F'\in\mathcal{F}$.
Then $L = B\cap S = B'\cap S$ with $B\in\mathcal{L}_F$ and $B'\in\mathcal{L}_{F'}$.
So $B$ and $B'$ are two blocks passing through the same line $L$.
The Steiner system property gives $B = B'$.
Hence $F = B+S = B' + S = F'$.
\end{proof}

Now let $\mathcal{L} = \bigcup_{F\in\mathcal{F}} \mathcal{L}_F$.

\begin{lemma}\label{lem:spread_partition}
The set $\mathcal{L}$ consists of $q^3 + q^2 + q + 1$ lines of $S$ and is partitionable into $q + 1$ line spreads of $S$.
\end{lemma}

\begin{proof}
	By Lemma~\ref{lem:LF_disjoint}, the sets $\mathcal{L}_F$ are pairwise disjoint, so $\mathcal{L}$ is a set of $\#\mathcal{F} \cdot \#D|^F = (q+1)(q^2 + 1) = q^3 + q^2 + q + 1$ lines in $S$ which admits a partition into the $q+1$ line spreads $\mathcal{L}_F$ with $F\in\mathcal{F}$.
\end{proof}

\begin{lemma}
\label{lem:pencil}
For each point $P$ of $S$, $\mathcal{L}|_P$ is a line pencil in the plane $E_P = \beta(P) \cap S$.
\end{lemma}

\begin{proof}
	Let $P$ be a point in $S$.

	By Lemma~\ref{lem:beta_flat}\ref{lem:beta_flat:poor}, the poor solid $S$ is not contained in the $5$-subspace $\beta(P)$.
	Therefore, $\dim(\beta(P)\cap S) \leq 3$.
	On the other hand, as both $S$ and $\beta(P)$ are contained in $H$, we have $\dim(\beta(P) + S) \leq \dim(H) = 6$ and therefore by the dimension formula $\dim(\beta(P)\cap S) = \dim(\beta(P)) + \dim(S) - \dim(\beta(P)+ S) \geq 3$.
	Hence $E_P = \beta(P) \cap S$ is a plane.

	Let $L\in \mathcal{L}|_P$.
	Then there is a block $B\in D|^H$ with $B\cap S = L$.
	By Lemma~\ref{lem:PBF}\ref{lem:PBF:PB}, $B \leq \beta(P)$.
	So $L = B \cap S \leq \beta(P) \cap S = E_P$.
	As the disjoint union of $q+1$ line spreads of $S$, $\mathcal{L}$ contains $q+1$ lines passing through $P$.
	Therefore, these lines form a line pencil in $E_P$ through $P$.
\end{proof}

\begin{lemma}
	\label{lem:gq}
	The incidence structure $(\qbinom{S}{1}{q},\mathcal{L}, \subseteq)$ is a projective generalized quadrangle of order $(s,t) = (q, q)$.
\end{lemma}

\begin{proof}
	Clearly, every line in $\mathcal{L}$ contains $q+1$ points in $S$.
	By Lemma~\ref{lem:spread_partition}, through every point in $S$ there pass $q + 1$ lines in $\mathcal{L}$.
	Now let $P$ be a point in $S$ and $L\in\mathcal{L}$ not containing $P$.

	By Lemma~\ref{lem:LF_disjoint}, there is a unique $F\in\mathcal{F}$ with $L\in\mathcal{L}_F$, and there is a line $L''\in\mathcal{L}_F$ passing through $P$.
	By Lemma~\ref{lem:pencil}, $L'' < E_P$, so we get $L \not < E_P$ as otherwise $L$ and $L''$ would be distinct intersecting lines in the spread $\mathcal{L}_F$.
	Moreover, $L$ and $E_P$ are both contained in $S$, so they cannot have trivial intersection.
	Therefore $L \cap E_P$ is a point.

	Now let $P'\in\qbinom{S}{1}{q}$ and $L'\in\mathcal{L}$ with $L \cap L' = P'$ and $P + P' = L'$.
	Then $L'$ is a line through $P$, so $L' < E_P$.
	So necessarily $P' = E_P \cap L$ and $L' = P + P'$, showing that $P'$ and $L'$ are unique.

	By Lemma~\ref{lem:pencil} indeed $L'\in\mathcal{L}$, as $P + P'$ is a line in $E_P$ containing $P$.
	This shows that $P'$ and $L'$ do always exist and therefore, the incidence structure $(\qbinom{S}{1},\mathcal{L})$ is a generalized quadrangle of order $(q,q)$.
\end{proof}

\begin{lemma}
	\label{lem:w}
	$(\qbinom{S}{1}{q},\mathcal{L})$ is isomorphic to $W(q)$.
\end{lemma}

\begin{proof}
	By Lemma~\ref{lem:gq} we know that $Q = (\qbinom{S}{1}{q},\mathcal{L}, \subseteq)$ is a finite generalized quadrangle of order $(s,t) = (q,q)$ embedded in $\PG(S)$.
	By the classification in~\cite[Theorem~1]{Buekenhout-Lefevre-1974-ArchMath25[1]:540-552} (see also \cite[4.4.8]{Payne-Thas-2009-FiniteGeneralizedQuadrangles2nd}), we know that $Q$ is a finite classical generalized quadrangle which are listed in \cite[3.1.2]{Payne-Thas-2009-FiniteGeneralizedQuadrangles2nd}.
	Comparing the orders and the dimension of the ambient geometry, the only possibility for $Q$ is the symplectic generalized quadrangle $W(q)$.
\end{proof}

Now we can proof our main result.

\begin{proof}[{{Proof of Theorem \ref{thm:partition}}}]
	Part~\ref{thm:partition:w} is a direct consequence of Lemma~\ref{lem:w} and Lemma~\ref{lem:spread_partition}.
	The equivalence of Part~\ref{thm:partition:w} and~\ref{thm:partition:q} has already been discussed in Section~\ref{subsect:gq}.
\end{proof}

Theorem~\ref{thm:alpha_points_block} is now a direct consequence.

\begin{proof}[{{Proof of Theorem \ref{thm:alpha_points_block}}}]
	We show that the statement in Theorem~\ref{thm:partition}\ref{thm:partition:q} is not satisfied.

	For $q$ prime, the only ovoids of $Q(4,q)$ are the elliptic quadrics $Q^{-}(3,q)$ \cite[Cor.~1]{Ball-Govaerts-Storme-2006-DCC38[1]:131-145}.
	As any two such quadrics have nontrivial intersection, there is no partition of $Q(4,q)$ into ovoids.

	For $q$ even, $Q(4,q)$ does not admit a partition into ovoids by \cite[3.4.1~(i)]{Payne-Thas-2009-FiniteGeneralizedQuadrangles2nd}.
\end{proof}

\section*{Acknowledgement}
I would like to thank Jan de Beule for his patient help with the theory of generalized quadrangles during a research stay at the Vrije Universiteit Brussel (VUB) in the framework of WOG Coding Theory and Cryptography.

\printbibliography
\end{document}